\documentclass{amsart}                 
\usepackage[T1]{fontenc}                   
\usepackage[utf8]{inputenc}                 
\usepackage[english]{babel}       
\usepackage{graphicx}                       % 
\usepackage{verbatim}
\usepackage{color}
\usepackage{picture}
\usepackage{amsmath,amsthm,amsfonts,amssymb}
\usepackage{mathtools}
\usepackage{bigints}
\usepackage{cleveref}

\newtheorem{thm}{Theorem}[section]
\newtheorem{lem}[thm]{Lemma}

\newtheorem{prop}[thm]{Proposition}

\theoremstyle{definition}
\newtheorem{rem}[thm]{Remark}
\theoremstyle{remark}

\newcommand{\R}{\mathbb{R}}
\newcommand{\N}{\mathbb{N}}

\newcommand{\de}{\partial}
\newcommand{\eps}{\varepsilon}

\DeclareMathOperator{\dint}{\displaystyle\int}

\title[nonlocal Poincaré-Wirtinger inequality]{Symmetry results for a nonlocal nonlinear Poincaré-Wirtinger inequality}
\author{Gianpaolo Piscitelli}
\address{Dipartimento di Scienze Economiche Giuridiche Informatiche e Motorie\\
Universit\`a degli Studi di Napoli Parthenope\\ 
Via Guglielmo Pepe, Rione Gescal, 80035 Nola (NA), Italy.}
\email{gianpaolo.piscitelli@uniparthenope.it}

\date{}

\begin{document}
%\doublespacing
\maketitle
%\tableofcontents
\begin{abstract}
\noindent In this paper, we study the optimal constant in the nonlocal nonlinear Poincar\'e-Wirtinger inequality in $(a,b)\subset\mathbb R$:
\begin{equation*}
\lambda_\alpha(p,q,r){\left(\int_{a}^{b}|u|^{q}dx\right)^\frac pq}\le{\int_{a}^{b}|u'|^{p}dx+\alpha\left|\int_{a}^{b}|u|^{r-2}u\, dx\right|^{\frac p{r-1}}},
\end{equation*}where $\alpha\in\mathbb R$, $p,q,r >1$ such that $\frac{2p}{p+2}\le q\le p$ and $\frac q2+1\le r \le
q+\frac q p$. 
This problem admits a variational characterization in the nonlocal setting, as the associated Euler-Lagrange equation involves an integral term depending on the unknown function over the entire interval of definition.
 
We prove the existence of a critical value $\alpha_C=\alpha_C (p,q,r)$ such that the minimizers are even and have constant sign for $\alpha\le\alpha_{C}$, while they are odd for $\alpha\geq \alpha_{C}$.

\end{abstract}

\section{Introduction}
Let $a,b, \alpha \in\R$ and $p,q,r>1$, in this paper we consider the following inequality:
\begin{equation}\label{non_local}
\begin{split}
&\lambda_\alpha(p,q,r) \left(\int_a^b|u|^q\text{d}x\right)^\frac pq\leq\int_a^b|u'|^p\text{d}x+\alpha\left|\int_a^b |u|^{r-2}u\text{d}x\right|^\frac p {r-1}\quad\forall u \in W^{1,p}(a,b)\\
&\qquad\qquad\qquad\qquad\qquad\qquad\qquad\qquad\qquad\qquad\qquad\qquad\qquad \text{s.t.}\ u(a)=u(b)=0.
\end{split}
\end{equation}

The optimal constant in \eqref{non_local} corresponds to the value realizing the minimum in the following eigenvalue problem
\begin{equation}
\label{operat}
\lambda_\alpha(p,q,r)=\inf \left\{ \mathcal Q_\alpha[u],\; u\in W_0^{1,p}(a,b),\,u\not\equiv 0 \right\},
\end{equation}
where 
\begin{equation}
\label{Qualpha}
\mathcal Q_{\alpha}[u]:=\dfrac{\dint_{a}^{b}|u'|^{p}dx+\alpha\left|\dint_{a}^{b}|u|^{r-2}u\, dx\right|^{\frac p{r-1}}}{\left(\dint_{a}^{b}|u|^{q}dx\right)^\frac pq}.
\end{equation}

This kind of problems lead in general to nonstandard associated Euler-Lagrange equations, known in literature as nonlocal, because they depend on the value that the unknown function assumes on the whole domain throughout the integral over $(a,b)$. Specifically,
\begin{equation*}%\label{eig_prob_non}
\left\{
\begin{array}{ll}
-(|y'|^{p-2}y')' + \alpha |\gamma|^{\frac p {r-1}-2}\gamma |y|^{r-2}=\lambda_\alpha(p,q,r)\, ||y||_q^{p-q}|y|^{q-2}y
\ \text{in}\ ]a,b[\\[.3cm]
y(a)=y(b)=0,
\end{array}
\right.
\end{equation*}
where $\gamma=\int_{a}^b |y|^{r-2}y\  dx$, except for some trivial cases detailed in \Cref{sec_eig}.

This type of nonlocal problems date back at least to the 1837 papers by Duhamel \cite{Du} and Liouville \cite{Lio} on thermo-elasticity. Since then, such problems have been studied in various contexts, including reaction-diffusion equations modeling chemical processes (see \cite{F,S}) or Brownian motion with random jumps (see \cite{P}). They have been the object of extensive research over the last thirty years \cite{F1, F, FV}, particularly through the study of variational problems involving a nonlocal one-parameter functional, both in the $n$-dimensional (\cite{BFNT}) and the one-dimensional (\cite{DP,DP2}) setting. 

In higher dimensions, problem \eqref{operat} has been treated (\cite{BFNT%,FP
}) only in the case when $p=q=r=2$. In that setting, the authors identify a {\it saturation phenomenon} under a volume constraint; they show that, among all domains with fixed measure, the optimal shape is the ball (for values of the parameter $\alpha$ below a critical one) and a union of two equal balls for supercritical values. A similar result holds when the Euclidean metric is replaced by a Finsler one \cite{Pi}. A related but distinct line of research, not investigated in this paper, concerns the study of symmetry-breaking phenomena, where the symmetry of optimal domains is lost below or above certain threshold values; see for instance \cite{BDNT,BCGM,Narxiv}.

In one dimension, the analysis of the nonlocal problem \eqref{operat} relies on the study of a generalized Wirtinger inequality, which was started by the pioneering work of Dacorogna, Gangbo, and Subia \cite{DGS} in the case $q\leq 2p$ and $r=2$. Subsequently,  in \cite{E,BKN,BK, N1, CD, GN} various parameter ranges have been analyzed. The question of symmetry/non symmetry of minimizers has been been completely settled in \cite{GGR}. Specifically, minimizers are symmetric when $q\leq (2r-1)p$, while no odd function can be a minimizer when $q>(2r-1)p$.

Our aim is to investigate the symmetry properties of the minimizers of \eqref{operat} and, as a consequence, to provide some informations on $\lambda_\alpha(p,q,r)$. To explore the full range of the exponents $p,q,r>1$, we recall that, to the best of our knowledge, the nonlocal problem \eqref{operat} has previously been studied only in the case $p=q=2$ with $2\leq r\leq 3$ in \cite{DP}, and for $p=q\ge2$ with $\frac p2+1\leq r\leq \frac p2$ in \cite{DP2}. In these regimes, the minimizers of \eqref{operat} are symmetric (either even or odd), and a saturation phenomenon occurs. In particular, for subcritical values of the parameter $\alpha$, the minimizers are even functions with constant sign, while for supercritical values, the minimizers are odd sign-changing. It is therefore natural to ask in which regions of the parameter space symmetry breaking is expected to occur.
In the present paper, we extend the range of treatable exponents and establish a new saturation phenomenon.

Throughout this paper, for the sake of simplicity, we will study the problem in the interval $(-1,1)$ instead of $(a,b)$. The general case can be easily recovered since the nonlocal eigenvalue admits the following rescaling:
\[
\lambda_{\alpha}(p,q,r;(a,b))=\left[ \left(\frac{2}{b-a}\right)^{\frac1{p'}+\frac 1q} \right]^p\ 
\lambda_{\tilde \alpha}\left(p,q,r;(-1,1)\right),
\]
with $\tilde \alpha=\left(\frac{b-a}{2}\right)^{\left(\frac 1{r-1}+\frac 1{p'}\right)p}\alpha$.

\begin{thm}
\label{mainthm1}
Let  $p,q,r >1$ be such that $\frac{2p}{p+2}\le q\le p$ and $\frac q2+1%\frac 12+\frac q2 +\frac q{2p}
\le r \le %p+1
q+\frac q p$.  Then there exists a positive number $\alpha_C=\alpha_C(p,q,r)$ such that:
\begin{itemize}
\item[{\it (i)}] if $\alpha<\alpha_{C}$, then $\lambda_\alpha(p,q,r)<\lambda_T (p,q,r)$,
\item[{\it (ii)}] If $\alpha\ge\alpha_{C}$, then $\lambda_\alpha(p,q,r)=\lambda_T (p,q,r)$,
\end{itemize}
where \[
\lambda_T(p,q,r) =\min_{\substack{u \in W^{1,p}_0(a,b)\\ \int_a^b |u|^{r-2}u\ \text{d}x=0\\ u\not\equiv 0 }}\frac{ \int_a^b|u'|^p\text{d}x}{\left(\int_a^b|u|^q\text{d}x\right)^\frac pq}.
\]
\end{thm}
In addition, we prove the following symmetry results for the solutions of problem \eqref{operat}. We refer to \Cref{sec_eig} for the definition of the generalized trigonometric function $\sin_{p,q}(\cdot)$.
\begin{thm}
\label{mainthm2} Let  $p,q,r >1$ be such that $\frac{2p}{p+2}\le q\le p$ and $\frac q2+1%\frac 12+\frac q2 +\frac q{2p}
\le r \le %p+1
q+\frac q p$. 
\begin{itemize}
\item[{\it (i)}] If $\alpha<\alpha_{C}$, then any minimizer $y$ of $\lambda_\alpha(p,q,r)$ is an even function with constant sign in $(-1,1)$.
\item[{\it (ii)}] If $\alpha>\alpha_{C}$, the function $y(x)=\sin_{p,q}(\lambda_T (p,q,r)
 x)$, $x\in(-1,1)$, is the unique minimizer, up to a multiplicative constant, of $\lambda_\alpha(p,q,r)$. Hence it is an odd function, $\int_{-1}^{1} |y|^{r-2}y\,dx=0$, and $\overline x=0$ is the only point in $(-1,1)$ such that $y(\overline x)=0$.
\item[{\it (iii)}]
%\item If both $q=1$, then $\alpha_{C}(p,1)=\frac{\pi^{2}}{2}$. Moreover, if $\alpha=\alpha_{C}$, there exists a positive minimizer of $\lambda(\alpha_C,p,1)$, and for any $\overline x\in]-1,1[$ there exists a minimizer $y$ of $\lambda(\alpha_C,p,1)$ which changes sign in $\overline x$, non-symmetric and with $\int_{-1}^{1}y(x)\,dx\ne 0$ when $\overline x\ne 0$.
If $\alpha=\alpha_C$, then $\lambda_{\alpha_C}(p,q,r)$ admits both a positive minimizer and the minimizer $y(x)=\sin_{p,q}(\pi_{p,q} x)$, up to a multiplicative constant. Moreover, if $r>\frac 12+\frac q2+\frac qp$ any minimizer has constant sign or it is odd. \end{itemize}
%Furthermore, if $r=p+1$, then $\alpha_{C}(p,q,p+1)=\frac{2^p-1}{2^p}\frac {q}{p'}\left(\frac {2p'} {p'+q}\right)^{1-\frac pq} \pi_{p,q}^p$.
\end{thm}

The outline of the paper follows. In \Cref{sec_unified}, we show how the treated nonlocal inequalities generalize classical Poincaré, Wirtinger, and Twisted inequalities; in \Cref{sec_eig}, we provide some recalls on the nonlocal eigenvalue problem we are dealing with; in \Cref{sec_H}, we study the properties of an auxiliary function useful to give some representation formulas of the eigenvalue and the eigenfunctions of problem \eqref{operat}; in \Cref{sec_proof}, we give the proof of the main Theorems.%; in Section \ref{sec_higher}, we discuss the behavior of higher nonlocal eigenvalues.

\section{A unified treatment of Poincar\'e, Wirtinger and Twisted inequalities}
\label{sec_unified}

One of the main aims of this paper is to unify and extend the study of Poincaré, Wirtinger, and Twisted inequalities, through the introduction of a penalization term. In this Section, we describe each of these classical inequalities in detail, explaining their connections with the nonlocal inequality under investigation.

Let $\lambda\in\R$, inequality \eqref{non_local} is a nonlinear generalization of the celebrated one-dimensional inequality
\begin{equation}
\label{classic}
\lambda\int_a^b|u|^2\text{d}x\leq \int_a^b|u'|^2\text{d}x\quad\forall u \in C^1(a,b);
\end{equation}
that is the Poincar\'e inequality when 
\begin{equation}
\label{classicP}
u(a)=u(b)=0
\end{equation}
and that is the Wirtinger inequality when 
\begin{equation}
\label{classicW}
\int_a^bu\ \text{d}x=0.
\end{equation}
The best constant $\lambda$  in both Poincar\'e \eqref{classic}-\eqref{classicP} and Wirtinger inequality \eqref{classic}-\eqref{classicW} is obtained for 
\[
\lambda=\left(\frac\pi{b-a}\right)^2.
\] 

The generalized Poincar\'e inequality (see e. g. \cite{GGR} and references therein) states that there exists a constant $\lambda_P(p,q)$ such that
\begin{equation}
\label{poincare}
\lambda_P(p,q) \left(\int_a^b|u|^q\text{d}x\right)^\frac pq\leq \int_a^b|u'|^p\text{d}x\quad\forall u \in W^{1,p}(a,b)\ \text{s.t.}\ u(a)=u(b)=0.
\end{equation}
It is easily seen that, when $\alpha=0$, the nonlocal inequality \eqref{non_local} is the Poincar\'e inequality \eqref{poincare}.
When $p=q=2$, we come back to the classical Poincar\'e inequality \eqref{classic}-\eqref{classicP}. Moreover, the optimal constant in \eqref{poincare} is also the minimum for the variational problem
\begin{equation*}%\label{eig_poincare}
\lambda_P(p,q) =\min_{\substack{u\in W^{1,p}_0(a,b)\\ u\not\equiv 0}}\frac{ \int_a^b|u'|^p\text{d}x}{\left(\int_a^b|u|^q\text{d}x\right)^\frac pq}
\end{equation*}
and the minimizing functions are even functions with constant sign. It is easily seen that $\lambda_P(p,q)$ is a homogeneous $(p,q)-$Dirichlet Laplacian eigenvalue (see e. g. \cite[Th. 3.3]{LE}). %\[\lambda_{P}(p,q)=\frac q{p'}\left(\frac{\pi_{p,q}}{b-a}\right)^q,\]

On the other hand, the generalized Wirtinger inequality states that there exists a constant $\lambda_W(p,q,r)$ such that 
\begin{equation}
\label{wirtinger}
\lambda_W(p,q,r) \left(\int_a^b|u|^q\text{d}x\right)^\frac pq\leq\int_a^b|u'|^p\text{d}x\quad\forall u \in W^{1,p}(a,b)\ \text{s.t.}\ \int_a^b |u|^{r-2}u\ \text{d}x=0.
\end{equation}
 When $p=q=r=2$, we come back to the classical Wirtinger inequality \eqref{classic}-\eqref{classicW}. 
%The best constant in the generalized Wirtinger inequality is obtained for \[\lambda_{W}(p,q,r)=\frac q{p'}\left(\frac{2\pi_{p,q}}{b-a}\right)^q.\]
Moreover, the optimal constant in \eqref{wirtinger} is also the minimum for the variational problem 
\[
\lambda_W(p,q,r) =\min_{\substack{u\in W^{1,p}(a,b)\\ \int_a^b |u|^{r-2}u\ \text{d}x=0\\ u\not\equiv 0 }}\frac{\int_a^b|u'|^p\text{d}x}{\left(\int_a^b|u|^q\text{d}x\right)^\frac pq}
\]
and the minimizing functions are odd functions. For the exact value of $\lambda_W(p,q,r)$ see \cite{GN}. It is easily seen that $\lambda_W(p,q,r)$ is a Neumann Laplacian eigenvalue (see e.g. \cite[Th. 3.4]{LE}). 

When both Dirichlet \eqref{classicP} and Neumann \eqref{classicW} boundary conditions hold, we speak of twisted boundary conditions \cite{BB, FH}. The best constant in the Twisted inequality \eqref{classic}-\eqref{classicP}-\eqref{classicW} is equal to 
\[
\lambda_T=\left(\frac{2\pi}{b-a}\right)^2.
\] 

The generalized Twisted inequality states that there exists a constant $\lambda_T(p,q,r)$ such that
\begin{equation}\label{twisted}
\begin{split}
&\lambda_T(p,q,r) \left(\int_a^b|u|^q\text{d}x\right)^\frac pq\leq\int_a^b|u'|^p\text{d}x\quad\forall u \in W^{1,p}(a,b)\\
&\qquad\qquad\qquad\qquad\qquad\qquad\qquad\qquad\qquad \text{s.t.}\ u(a)=u(b)=0\ \text{and}\ \int_a^b |u|^{r-2}u\text{d}x=0.
\end{split}
\end{equation}
It is easily seen that, meanwhile when $\alpha\to + \infty$, inequality \eqref{non_local} tends to the Twisted inequality \eqref{twisted}.
When $p=q=r=2$, we come back to the classical Twisted inequality \eqref{classic}-\eqref{classicP}-\eqref{classicW}. 

Moreover, the optimal constant in \eqref{twisted} is also the minimum for the variational problem
\begin{equation}
\label{eigenT}
\lambda_T(p,q,r) =\min_{\substack{u \in W^{1,p}_0(a,b)\\ \int_a^b |u|^{r-2}u\ \text{d}x=0\\ u\not\equiv 0 }}\frac{ \int_a^b|u'|^p\text{d}x}{\left(\int_a^b|u|^q\text{d}x\right)^\frac pq}
\end{equation}
and the minimizing functions are odd functions. See \cite[Thm. 1.1]{CD} for the exact value of $\lambda_T(p,q,r)$, and note that there is no dependence on the parameter $r$.

\section{The eigenvalue problem}
\label{sec_eig}
In this Section, we firstly recall some results on the generalized trigonometric functions and then some properties of the eigenvalue problem \eqref{operat}.

\subsection{The $(p,q)-$circular functions} 
We briefly summarize some properties of the $p$-trigonometric functions for any fixed $1<p<+\infty$ (refer e. g. \cite{LE, L, Pe}). These functions generalize the familiar trigonometric functions and coincide with them when $p=2$.

Let us consider the function $F_p:[0, 1]\to\R$ defined as
\begin{equation*}
F_p(x)=\bigintsss_0^x\frac{dt}{\left(1-{t^p}\right)^\frac 1p}.
\end{equation*}
Denote by $z(s)$ the inverse function of $F$ which is defined on the interval $\left[0,\frac{\pi_p}{2}\right]$, where
\[
\pi_p=2\bigintsss_{ 0}^{1}\frac{dt}{\left(1-{t^p}\right)^\frac 1p}.\]
Therefore, the $p$-sine function $\sin_p$ is defined as the following periodic extension of $z(t)$:
\[
\sin_p(t)=\left\{ \begin{split}
& z(t) &&\text{if}\ \ t\in\left[0,\frac{\pi_p}{2}\right],\\
& z(\pi_p-t) &&\text{if}\ \ t\in\left[\frac{\pi_p}{2}, \pi_p\right], \\
& -\sin_p(-t)\  &&\text{if}\ \ t\in\left[-\pi_p, 0\right]. \\
\end{split}
\right.
\]
It is extended periodically to all $\R$, with period $2\pi_p$. Furthermore, the $p$-cosine function is defined by
\[
\cos_p( t) =\frac{d}{dt}\sin_p \left( t\right)
\]
and is a $2 \pi_p$-periodic and odd function.

%Now let us consider the problem of minimizing the Rayleigh quotient\[\mathcal Q_p(u)=\frac{\ds\int_{-1}^1 |u'(x)|^p \ dx}{\ds\int_{-1}^1 |u(x)|^p \ dx} \qquad (1<p<\infty),\] among all real valued functions $u\in W_{0}^{1,p}(-1,1)$. This minimum is attained and is positive. This means that it is an eigenfunction of the $1$-dimensional Dirichlet $p$-Laplacian eigenvalue problem:\begin{equation*}\left\{\begin{array}{ll}-(|y'|^{p-2}y')' =\lambda^{(p)}\, |y|^{p-2}y& \text{in}\ ]-1,1[\\[.4cm]y(-1)=y(1)=0.\end{array}\right.\end{equation*}Furthermore, the first eigenvalue $\lambda_1^{(p)}$ is $\left(\frac{\pi_p}{2}\right)^p$ and the first eigenfunction is $\sin_p(\pi_p x)$, up to a multiplicative constant.

To further extend the definitions of trigonometric functions, let us consider $p,q>1$ and set 
\[
\pi_{p,q}:=2\bigintsss_{0}^1\frac 1 {(1-t^q)^\frac1p}dt=\frac 2 q B\left(\frac 1 {p'},\frac 1 q\right)=\frac 2 q \frac{\Gamma \left(\frac 1 {p'}\right)\Gamma \left(\frac 1q \right)}{\Gamma \left(\frac 1{p'}+\frac 1 q\right)},
\]
where $B$ and $\Gamma$ are the beta and the gamma function, respectively.

This definition coincides with $\pi_p$ when $p=q$. %we have $\pi_{p,q}=\frac2q B\left(\frac1{p'}, \frac1q\right)$. 
Therefore the function $\sin_{p,q}$ is defined on the interval $[0,\frac{\pi_{p,q}}2]$ as the inverse of $F_{p,q}:[0,1]\to\R$ given by 
\[
F_{p,q}(x)=\bigintsss_0^x\frac 1{(1-t^q)^\frac 1p}dx
\]
 and extended to the real line by the usual process involving the symmetry and the $2\pi_{p,q}$ periodicity.

Finally, we recall from \cite[Thm. 3.3]{LE}, that any eigenvalue and eigenfunction of the $1$-dimensional Dirichlet $(p,q)-$Laplacian eigenvalue problem:
\begin{equation}
\label{eig_pq_1D}
\left\{
\begin{array}{ll}
-(|y'|^{p-2}y')' =\lambda |y|^{q-2}y& \text{in}\ ]-1,1[\\[.4cm]
y(-1)=y(1)=0.
\end{array}
\right.
\end{equation}
are of the form 
\[
\lambda_n=c_1\frac{q}{p'}\left(\frac{n \pi_{p,q} }2\right)\quad\text{and}\quad y_n(x)=c_2\sin_{p,q}\left(\frac{n \pi_{p,q}}2(x+1)\right)\quad \forall n\in\N
\]
respectively, for $c_1,c_2\in\R$. Clearly, when $p=q$, we fall in the case of the $p$-Laplacian problem.
% This means that the first eigenfunction is positive and minimize the problem \eqref{eig_poincare} in $[-1,1]$.

\subsection{The eigenvalue problem}
We firstly show some properties of the solution of the eigenvalue problem \eqref{operat}.
\begin{prop}
\label{propr}
%Propriet\`a di $\lambda(\alpha,r)$: Lipschitz, monotonia, limite... 
Let $\alpha\in \R$, $p,q,r >1$ be such that $r \le p+1$. Then, problem \eqref{operat} admits a solution in $W_0^{1,p}(-1,1)$ and any minimizer $y$ of \eqref{operat} is a solution of the following Dirichlet homogeneous  problem
\begin{equation}
\label{el}
\left\{
\begin{array}{ll}
-(|y'|^{p-2}y')' + \alpha |\gamma|^{\frac p {r-1}-2}\gamma |y|^{r-2}=\lambda_\alpha(p,q,r)\, ||y||_q^{p-q}|y|^{q-2}y& \text{in}\ ]-1,1[\\[.4cm]
y(-1)=y(1)=0,
\end{array}
\right.
\end{equation}
where 
\[
\gamma=
\begin{cases}
0 \text{ if both } r=p+1 \text{ and }\displaystyle \int_{-1}^1 |y|^{r-2}y\  dx=0, \\
%\displaystyle\left|\ds\int_{-1}^1 |y|^{r-2}y\  dx\right|^{\frac{p}{r-1}-2}\left(
\displaystyle \int_{-1}^1 |y|^{r-2}y\  dx% \right) 
\text{ otherwise}.
\end{cases}
\]
%(we mean $\gamma=0$ if $r=2$ and $\int_{-1}^1 y|y|\  dx=0$).

Moreover, $y, y'|y'|^{p-2}\in C^1[-1,1]$.
\end{prop}
\begin{proof} 
Standard methods of Calculus of Variations prove the existence of a minimizer.
%Let us observe that, since $p\geq q \geq r-\frac qp \geq r-1$, we have that $p\geq r-1$. 
If $p>r-1$, the functional $\mathcal Q_\alpha[\cdot]$ in \eqref{Qualpha} is differentiable and hence the associated Euler-Lagrange equation leads to \eqref{el}; meanwhile, when $p=r-1$, the problem \eqref{operat} coincides with problem \eqref{eigenT}; hence $\gamma=0$ and we get the conclusion.

Finally, the fact that $y, y'|y'|^{p-2}\in C^1[-1,1]$ is easily seen from \eqref{el}. 
\end{proof}

At this stage, we analyze the monotonicity and asymptotic properties of the eigenvalue \eqref{operat} with respect to the parameter $\alpha$.
\begin{prop}
\label{eig_prop}
For any fixed $p,q,r >1$ such that $r \le p+1$, the function $\alpha\in\R\mapsto\lambda_{\alpha}(p,q,r)$ is Lipschitz continuous, nondecreasing with respect to $\alpha\in\R$ and 
\[
\lim_{\alpha\to -\infty}\lambda_\alpha(p,q,r)=-\infty,\quad
\lim_{\alpha\to +\infty}\lambda_\alpha(p,q,r)=\lambda_T(p,q,r).
\]

\end{prop}
\begin{proof}
Let us fix $\varepsilon>0$, then by using the H\"older inequality with coniugate exponents $\frac q{r-1}$ and $\frac{q}{q-r+1}$, we have
\begin{equation*}
\mathcal{Q}_{\alpha+\eps} [u]\le\mathcal{Q}_\alpha[u]+\eps\frac{\left(\dint_{-1}^{1} |u|^{r-1}\  dx\right)^{\frac p{{r-1}}}}{\left(\dint_{-1}^{1} |u|^q\  dx\right)^\frac pq }\leq\mathcal{Q}_\alpha[u]+ 2^\frac{p(q-r+1)}{q(r-1)}\eps.
\end{equation*}
Therefore, we gain the following chain of inequalities
\begin{equation*}
\mathcal{Q}_\alpha[u]\leq\mathcal{Q}_{\alpha+\eps}[u]\le\mathcal{Q}_\alpha[u]+ 2^\frac{p(q-r+1)}{q(r-1)}\eps\quad\forall \ \varepsilon>0.
\end{equation*}
By taking the minimum for any $u\in W_0^{1,p}(-1,1)$, we have
\begin{equation*}
\lambda_\alpha (p,q,r)\leq\lambda_{\alpha+\varepsilon} (p,q,r)\leq\lambda_ \alpha (p,q,r)+2^\frac{p(q-r+1)}{q(r-1)}\varepsilon\quad\forall\varepsilon>0,
\end{equation*}
that implies the desired Lipschitz continuity and monotonicity. 

Now, let us fix a positive admissible function $\varphi\in W^{1,p}_0(-1,1)$. Then, we have that $\mathcal Q_{\alpha}[\varphi] \to -\infty \quad\text{as }\alpha\to -\infty$ and, since $\lambda_\alpha(p,q,r) \le \mathcal Q_{\alpha}[\varphi]$, we have that
\[
\lim_{\alpha\to -\infty}\lambda_\alpha(p,q,r)=-\infty.
\]

Finally, let us consider a sequence $\{\alpha_n\}_{n\in\N}\to+\infty$.
Since $\lambda_\alpha ( p,q, r)$ is nondecreasing with respect to $\alpha$, we have that $\lambda_\alpha(p,q,r)\le\lambda_T(p,q,r)$ for any $\alpha\in\R$. Let us denote $u_{n}=u_{\alpha_n}$ the normalized ($||u_{n}||_{L^q}=1$) minimizer in $W_0^{1,p}$ of \eqref{operat} when the value of the parameter is $\alpha_n$; we have that
\begin{equation*}
\lambda_{\alpha_n} (p,q,r)=\int_{-1}^{1} |u'_n|^p\  dx + \alpha_n\left(\int_{-1}^{1} |u_n|^{r-2}u_n \  dx\right)^\frac{p}{r-1}\leq\lambda_T(p,q,r).
\end{equation*}
This implies that, up to a subsequence, $u_n$ strongly converges in $L^p(-1,1)$ and weakly in $W_0^{1,p}(-1,1)$ to a function $u\in W_0^{1,p}(-1,1)$ such that $\|u\|_{L^q}=1$. Furthermore, we have that 
\begin{equation*}
\left( \int_{-1}^{1} |u_n|^{r-2}u_n \  dx\right)^\frac{p}{r-1}\leq\frac{ \lambda_T(p,q,r)}{\alpha_n}\rightarrow 0 \quad\text{as}\ n\rightarrow + \infty
\end{equation*}
which means that $\int_{-1}^1 |u|^{r-2}u\  dx=0$. On the other hand, % the lower semicontinuity of the functional in $W_0^{1,p}(-1,1)$ implies that \begin{equation*} \int_{-1}^{1} |u'|^p\  dx \leq \liminf_{n \rightarrow +\infty}\int_{-1}^{1} |u'_n|^p\  dx. \end{equation*} Therefore, 
 since $u$ is an admissible function for \eqref{eigenT}, by using the lower semicontinuity of the integral, we have that
\begin{equation*}
\begin{split}
\lambda_T(p,q,r)\le
\int_{-1}^{1} |u'|^p\  dx &\leq \liminf_{n \rightarrow +\infty}\left[\int_{-1}^{1} |u'_n|^p\  dx
+ \alpha_n\left( \int_{-1}^{1} |u_n|^{r-2}u_n \  dx\right)^\frac{p}{r-1}\right]\\
&=\lim_{n \rightarrow + \infty}\lambda_{\alpha_n}(p,q,r)\leq\lambda_T(p,q,r)
\end{split}
\end{equation*}
and hence the conclusion follows.% that\[\lim_{\alpha\to +\infty}\lambda_\alpha(p,q,r)= \lambda_T(p,q,r).\]
\end{proof}

\section{The auxiliary function $H$}
\label{sec_H}
In this Section, we study the behavior of an auxiliary function on which is based the proof of the main results (\Cref{mainthm1} and \Cref{mainthm2}). We consider the following integral function:
\[ H(m,p,q,r):(m,p,q,r)\in[0,1]\times]1,+\infty[\times]1,+\infty[\times]1,+\infty[\mapsto\R,\]
defined as 
\begin{multline}
\label{Hdef}
H(m,p,q,r):= \bigintsss_{-m}^1\frac{dy}{ [1- R(m,q,r)(1- |y|^{r-2}y) - |y|^q]^\frac1p}\\
=\bigintsss_{0}^1  \frac{dy}{ [1-R(m,q,r)(1-y^{r-1}) - y^q]^\frac1p}  +\bigintsss_{0}^1\frac{mdy}{ [1-R(m,q,r)(1+m^{r-1} y^{r-1}) -m^q y^q]^\frac1p}
\end{multline}
where 
\begin{equation}\label{Rdef}
R(m,q,r)=\frac{1-m^q}{1+m^{r-1}}.
\end{equation}

It will also be very useful in the sequel to consider $h$, the integrand function of $H$, that is defined as
\begin{multline*}
h(m,p,q,r,y):=\\ \frac{1}{ [1-R(m,q,r)(1-y^{r-1}) - y^q ]^\frac 1 p}
+\frac{m}{ [1-R(m,q,r)(1+ m^{r-1}y^{r-1})- m^qy^q ]^\frac 1p},
\end{multline*}
for any $y\in[0,1[$, except when $m=y=0$.

We will prove the monotonicity of the auxiliary function with respect to $r$ (more precisely the function $h$) in \Cref{prop-monotonia-r} and then with respect to $m$ in \Cref{Hdecr}. Finally, in \Cref{prop_stime_H}, we provide some useful estimates for the function $H$.
\begin{lem}
\label{prop-monotonia-r}
Let $p,q,r>1$ such that $\frac 12+\frac q2 +\frac q{2p}\le r \le q+\frac q p$.  For any fixed $y\in[0,1[$ and
\begin{itemize}
\item for any fixed $m\in[0,1[$, the function $h(m,p,q,r,y)$ is strictly increasing with respect to $r$;
\item for $m=1$, the function $h(1,p,q,r,y)$ is constant with respect to $r$. 
\end{itemize}
\end{lem}
\begin{proof} We divide the proof into three steps: in the first step we compute the expression of the derivative of $h$ with respect to $r$ for any $m\in]0,1[$ and $y\in]0,1[$; in the second step we study the sign of the aforementioned derivative; in the third step we analyze the cases excluded by the previous steps. From now on, for the sake of simplicity, we will set  $R=R(m,q,r)$.

{\it Step 1 (The derivative of $h$).} Let us start by considering the case when $m\in]0,1[$ and $y\in]0,1[$.
Differentiating $h$ with respect to $r$, we have
\begin{equation*}
\begin{split}
\de_{r}h (m,p,q,r,y)= &- \frac 1 p \frac {(1-y^{r-1})\de_{r} R + R\, y^{r-1}\log y}{\left[1-R(1- y^{r-1}) - y^q\right]^\frac{p+1}p} +\\[.2cm]
&- \frac m p \frac{
[-(1+m^{r-1}y^{r-1})\de_{r}R- R\,m^{r-1}y^{r-1}(\log m +\log y)]}{ \left[1-R(1+m^{r-1}y^{r-1})-m^{q}y^{q}\right]^\frac {p+1} p}.
\end{split}
\end{equation*}
%\begin{equation*}\begin{split}\de_{r}h = &- \frac {1}{pF_{I}^{p+1}} \big[-(1-y^{r-1})\de_{r} R + R\, y^{r-1}\log y\big]+\\[.2cm]&- \frac {n}{pF_{I\!I}^{p+1}} \big[-(1+n^{r-1}y^{r-1})\de_{r}R- R\,n^{r-1}y^{r-1}(\log n +\log y)\big],\end{split}\end{equation*}
Therefore, in order to compute the derivative of $h$ with respect to $r$, we need to differentiate $R$ (defined in \eqref{Rdef}) with respect to $r$. We have
\[
\de_{r}R= -\frac{1-m^{q}}{(1+m^{r-1})^{2}}m^{r-1}\log m
\]
and hence
\begin{equation}
\begin{split}
\label{derivata_h}
\de_{r}h (m,p,q,r,y)= &-\frac1p\frac{1-m^{q}}{(1+m^{r-1})^{2}}\bigg\{\frac{ (1-y^{r-1})m^{r-1}\log m +y^{r-1}(1+m^{r-1})\log y}{\left[1-R(1- y^{r-1}) - y^q\right]^\frac{p+1}p}+ \\[.3cm]
& +m^r \frac{(1+m^{r-1}y^{r-1})\log m-(1+m^{r-1})y^{r-1}(\log m +\log y)}{ \left[1-R(1+m^{r-1}y^{r-1})-m^{q}y^{q}\right]^\frac {p+1} p}\bigg\}.
\end{split}
\end{equation}
{\it Step 2 (The monotonicity of $h$).}
It is easily seen that the numerator of the first ratio, in the curly brackets of \eqref{derivata_h}, is negative. If the numerator of the second ratio is also negative, we get the desired monotonicity. %\begin{equation}\label{segno_der_h}\partial _r h(n,p,q,r,y)>0.\end{equation}
Otherwise, if this second numerator is positive, let us observe that
\begin{equation*}
m^q(1-R(1- y^{r-1}) - y^q)\leq [1-R(1+m^{r-1}y^{r-1})-m^{q}y^{q}],
\end{equation*}
%\begin{equation*}\label{firint}F_{I}(n,p,q,r,y):=\displaystyle \left[1-R(1- y^{r-1}) - y^q\right]^\frac 1 p \le \left[1-y^q\right]^\frac 1p,\end{equation*}and \begin{equation*}\label{secint}F_{I\!I}(n,p,q,r,y):=\displaystyle \left[1-R(1+n^{r-1}y^{r-1})-n^{q}y^{q}\right]^\frac 1 p\ge n^q\left(1-y^{q}\right)^\frac 1p,\end{equation*}
%Resuming, we have\begin{multline*}\de_{r}h = \frac1p\frac{1-n^{q}}{(1+n^{r-1})^{2}}\bigg\{ \overbrace{\bigg[ -(1-y^{r-1})n^{r-1}\log n -y^{r-1}(1+n^{r-1})\log y\bigg]}^{h_{1}(n,r,y)}\frac{1}{F_{I}^{p+1}}+ \\[.3cm]+ \underbrace{\bigg[-(1+m^{r-1}y^{r-1})\log n+(1+n^{r-1})y^{r-1}(\log n +\log y)\bigg]}_{h_{2}(n,r,y)}\frac {n^{r}}{F_{I\!I}^{p+1}}\bigg\}.\end{multline*}
%By \eqref{firint} and \eqref{secint} 
that implies:
\begin{multline}\label{derivata_hge}
\de_{r}h(m,p,q,r,y) \ge -\frac1p\frac{1-m^{q}}{(1+m^{r-1})^{2}}\bigg\{\frac{ (1-y^{r-1})m^{r-1}\log m +y^{r-1}(1+m^{r-1})\log y}{m^{\frac {-q(p+1)}p}  \left[1-R(1+m^{r-1}y^{r-1})-m^{q}y^{q}\right]^\frac {p+1} p }+ \\[.3cm]
+m^r \frac{(1+m^{r-1}y^{r-1})\log m-(1+m^{r-1})y^{r-1}(\log m +\log y)}{ \left[1-R(1+m^{r-1}y^{r-1})-m^{q}y^{q}\right]^\frac {p+1} p}\bigg\}.
\end{multline}
%\begin{equation*}\begin{split}\de_{r}h \ge \frac1p\frac{1-n^{q}}{(1+n^{r-1})^{2}}\bigg\{ &\bigg[ -(1-y^{r-1})n^{r-1}\log n -y^{r-1}(1+n^{r-1})\log y\bigg]\frac{1}{(1-y^{p})^{\frac {p+1}{p}}}+ \\[.2cm]+& \bigg[(y^{r-1}-1)\log n+(1+n^{r-1})y^{r-1}\log y\bigg]\frac {n^{r-1-q}}{(1-y^{p})^{\frac{p+1}{p}}}\bigg\}.\end{split}\end{equation*}
Hence, by setting
\begin{multline*}
g(m,p,q,r,y):=\bigg[ -(1-y^{r-1})m^{r-1}\log m -y^{r-1}(1+m^{r-1})\log y\bigg]+\\
+ \bigg[(y^{r-1}-1)\log m+(1+m^{r-1})y^{r-1}\log y\bigg]m^{r-\frac{q(p+1)}p},
\end{multline*}
we have that inequality \eqref{derivata_hge} can be written as
\begin{equation}\label{derivative_h}
\de_{r}h (m,p,q,r,y) \ge \frac1p\frac{1-m^{q}}{(1+m^{r-1})^{2}m^{\frac {q(p+1)}p}} g(m,p,q,r,y).
\end{equation}
To prove the positivity of $\partial_r h$, we will show that
\begin{equation}
\label{segno_der_g}
g(m,p,q,r,y)>0,
\end{equation}
by proving that $g$ is decreasing with respect to $y$ in the interval $]0,1[$, since $g(m,p,q,r,1)=0$. By differentiating $g$ with respect to $y$, we obtain
\begin{equation*}
\begin{split}
\de_{y}g (m,p,q,r,y)& = \bigg[(r-1)y^{r-2}m^{r-1}\log m-(r-1)y^{r-2}(1+m^{r-1})\log y-y^{r-2}(1+m^{r})\bigg]\\
&\ \quad +\bigg[(r-1)y^{r-2}\log m+(1+m^{r-1})((r-1)y^{r-2}\log y+y^{r-2})\bigg]m^{r-\frac{q(p+1)}p}\\
& =y^{r-2}\bigg[(r-1) (m^{r-1} +m^{r-\frac{q(p+1)}p})\log m+(r-1)(1+m^{r-1})\left(m^{r-\frac{q(p+1)}p}-1\right)\log y  \\
&\ \quad +(1+m^{r-1})\left(m^{r-\frac{q(p+1)}p}-1\right)\bigg].
\end{split}
\end{equation*}
Since $r-\frac{q(p+1)}p\le  0$ by assumptions, this derivative is negative if and only if
%\begin{equation*}(1+n^{r-1})(n^{r-1-q}-1)((r-1)\log z+1) < (1-r)(n^{r-1}+n^{r-1-q})\log n.\end{equation*}The above inequality is true, as we will show that 
\begin{equation}
\label{ineqlog}
\log y < -\frac{\left(m^{r-1}+m^{r-\frac{q(p+1)}p}\right)\log m}{\left(1+m^{r-1}\right)\left(m^{r-\frac{q(p+1)}p}-1\right)}-\frac1{r-1}. %=:-\frac1r+\ell(n,q,r). 
\end{equation}
Since the left-hand term is negative, then if the right-hand side of \eqref{ineqlog} is nonnegative, then the inequality \eqref{ineqlog} holds.
%{\bf Claim 2.} {\em For any $q\in [p,+\infty[$, $r\in \left[\frac q2+1,q+1\right]$ and $n\in]0,1[$, $\ell(n,q,r)> \frac1{r-1}$.} We will show that\[\ell(n,q,r) >  \frac 1{r-1}.\]We have\[\ell(n,q,r)= \frac{(n^{r-1}+n^{r-1-q})\log n }{(1+n^{r-1})(1-n^{r-1-q})}> \frac 1{r-1}\]
To this aim, we will equivalently show that
\begin{multline}\label{segno_f}
f(m,p,q,r):=-\left(m^{r-1}+m^{r-\frac{q(p+1)}p}\right)\log m - \frac1{r-1}\left(1+m^{r-1}\right)\left(m^{r-\frac{q(p+1)}p}-1\right) > 0.
\end{multline}
We have
\begin{equation*}
\begin{split}
f(m,p,q,r)&=\left(m^{r-1}+m^{r-\frac{q(p+1)}p}\right)\log\frac{1}{m} +\frac1{r-1} \left( 1+m^{r-1}-m^{r-\frac{q(p+1)}p}-m^{2r-1-\frac{q(p+1)}p}\right)\\
&= m^{r-1}\left(\log\frac 1m+\frac1{r-1} \right)+ m^{r-\frac{q(p+1)}p}\left(\log\frac1m-\frac1{r-1} \right)+\frac1{r-1} \left(1-m^{2r-1-\frac{q(p+1)}p}\right)\\
&\ge m^{r-1}\left(\log\frac 1m+\frac1{r-1}\right)+ m^{r-\frac{q(p+1)}p}\left(\log\frac1m-\frac1{r-1}\right)\\
&=m^{r-\frac{q(p+1)}p}\left(m^{\frac{q(p+1)}p-1}\left(\log\frac 1m+\frac1{r-1}\right)+\log\frac 1m-\frac1{r-1}\right).
\end{split}
\end{equation*}
because $2r-1-\frac{q(p+1)}p\ge0$.
Hence, the positivity of $f$ follows if
\begin{equation}
\label{segno_e}
e(m,p,q,r):=m^{\frac{q(p+1)}p-1}\left(\log\frac 1m+\frac1{r-1}\right)+\log\frac 1m-\frac1{r-1}>0.
\end{equation}
Since $e(1,p,q,r)=0$, to prove \eqref{segno_e}, we show that $e$ is decreasing with respect to $m$; indeed, we have
\[
\partial_m e (n,p,q,r)=m^{\frac{q(p+1)}p-2}\left(\log\left( \frac 1{m^{\frac{q(p+1)}p-1}} \right)+\frac {\frac{q(p+1)}p-1}{r-1} -1-\frac1{m^{\frac{q(p+1)}p-1}} \right)
\]
that is negative since $\log z < z-1$ when $z>1$, $m^{\frac{q(p+1)}p-1}<1$ and $r\ge\frac 12+\frac q2 +\frac q{2p}$, by assumptions.

Hence \eqref{segno_e}, \eqref{segno_f}, \eqref{ineqlog} and \eqref{segno_der_g} are satisfied, and recalling the behavior of $h$ from \eqref{derivative_h}, this implies that
\[
\de_{r}h (m,p,q,r,y)\ge \frac1p\frac{1-m^{q}}{(1+m^{r-1})^{2}} g(m,p,q,r,y)> \frac1p\frac{1-m^{q}}{(1+m^{r-1})^{2}} 
g(m,p,q,r,1)=0.
\]
when $m\in]0,1[$ and $y\in]0,1[$.
%Finally, we show that if $m<1$ then $\de_{r}h(r,m,y)>0$ for $m\in ]0,1[$ and $y\in]0,1[$. This follows from the third claim below.
%
%{\bf Claim 3.} {\em If $0<m<1$ and $1<r\le 2$, then $g(r,m,\cdot$) is strictly decreasing for $y\in]0,1[$}.
%
%Taking a closer look to the inequality \eqref{ineqlog},
%
% we have that if $m<1$ the inequality \eqref{ineq0} is strict.

\textit{Step 3 (The trivial cases).} We observe that if $m=0$, then $R=1$ and 
\[h(0,p,q,r,y)=\frac{1}{ (y^{r-1}- y^q)^\frac 1 p},
\]
that is strictly increasing with respect to $r$.

Meanwhile, if $m=1$, then $R=0$ and
\[
h(1,p,q,r,y)=\frac{2}{ (1- y^q)^\frac 1 p}, 
\]
that is constant with respect to $r$. 

Finally, when $y=0$, we have
\[
h(m,p,q,r,0)=\frac{1+m}{1-R},
\]
that is strictly increasing with respect to $r$.
\end{proof}

At this stage, to prove the monotonicity of $H$ with respect to $m$, we argue using a change of variables similarly as in \cite{GGR}. %Before providing the result, let us explicitly note that, in the previous Lemma, we have only supposed that $q\le p$ and $\frac q2+1\le r\le q+\frac q p$. These two conditions imply that $q\ge \frac{2p}{p+2}$ but, for the following result, we need to suppose a bit more: $q\geq \frac 45 p$, that is also the assumption we use to prove the main Theorems.
%From this following treatment is excluded the case when $y=0$, for which we have\[h(n,p,q,r,0)=\frac{1+n}{[1-R(n,q,r) ]^\frac 1 p},\]that is strictly increasing with respect to $n$. 
\begin{lem}
\label{Hdecr} Let $p,q>1$ be such that $q\le p$, then $\partial_mH(m,p,q,\frac q2 +1)\geq 0$ for any $m\in]0,1[$.
\end{lem}
\begin{proof} 
Let us explicitly observe that, when $r=\frac q2+1$, then
\[
R\left(m,q,\frac q2+1\right)=1-m^\frac q2,\quad\forall \ {m\in[0,1]}.
\]
Hence, for any fixed $p$ and $q$ in the assumed ranges, we can write 
\begin{equation}
\label{Kdef}
K(m):=H\left(m,p,q,\frac q2+1\right)=\int_0^1 \left(A(m,y)^{-\frac 1p}+mB(m,y)^{-\frac 1p}\right)dy,
\end{equation}
where
\begin{align*}
&A(m,y):= m^\frac q 2 +(1-m^\frac q2)y^{\frac q2}-y^q,\quad\forall\ {(m,y)\in[0,1]^2},\\
&B(m,y):= m^\frac q 2-(1-m^\frac q2)m^\frac q 2 y^\frac q2-m^q y^q,\quad\forall \ {(m,y)\in[0,1]^2}.
\end{align*}

We get
\begin{multline*}
K'(m)=-\frac 1 p\int_0^1 \left(A(m,y)^{-\frac 1p-1}\partial_m A(m,y)+ B(m,y)^{-\frac 1p-1}\left(-p B(m,y)+m\partial_m B(m,y)\right)\right)dy.
\end{multline*}

Differentiating $A(m,y)$ and $B(m,y)$ with respect to $m$, we obtain
\begin{align*}
\partial_m A(m,y) & =\frac q 2 m^{\frac q2 -1}(1-y^{\frac q2})\\
-pB(m,y)+m B_m(m,y) & =\left(\frac q 2-p\right) m^{\frac q2}(1-y^{\frac q2})+\left(q -p\right) m^{q}(y^{\frac q2}-y^q).
\end{align*}
Hence we have
\begin{equation*}
\begin{split}
K'(m)&= m^{\frac q2-1}\int_0^1 -\frac q {2p} \frac{1-y^{\frac q2}}{A(m,y)^{\frac 1p+1}} +\left(1 -\frac q{2p}\right)\frac{m(1-y^{\frac q2})}{B(m,y)^{\frac 1 p +1}}+\left(1-\frac{q}{p}\right)\frac{m^{\frac q2 +1}(y^{\frac q2}-y^q)}{B(m,y)^{\frac 1 p +1}}\ dy.\\
%&=-\frac q pn^{\frac q2-1}\int_0^1 \left( \frac{1-y^{\frac q2}}{A(n,y)^{\frac 1p+1}} -\frac{n(1-y^{\frac q2})}{B(n,y)^{\frac 1 p +1}}\right) dy\\&\qquad\qquad\qquad\qquad\qquad\qquad+\left(1-\frac q{p}\right)n^{\frac q2-1}\int_0^1 \left(\frac{n(1-y^{\frac q2})}{B(n,y)^{\frac 1 p +1}}+\frac{n^{\frac q2 +1}(y^{\frac q2}-y^q)}{B(n,y)^{\frac 1 p +1}}\right)dy.
\end{split}
\end{equation*}
%It is easily seen that the second integral is nonnegative, meanwhile, t
To prove the nonnegativity of the integral, we have to show that
\begin{equation}
\label{Gobbino_int}
-\frac q {2p}\int_0^1 \frac{1-y^{\frac q2}}{A(m,y)^\frac{p+1}{p}}\ dy +
\int_0^1\left[  \left(1 -\frac q{2p}\right)\frac{m(1-y^{\frac q2})}{B(m,y)^{\frac 1 p +1}}+\left(1-\frac{q}{p}\right)\frac{m^{\frac q2 +1}(y^{\frac q2}-y^q)}{B(m,y)^{\frac 1 p +1}}\right]\ dy\ge 0.
\end{equation}
Following the ideas of \cite{GGR}, for all $m\in (0,1)$ and $z\in (0,1)$, we set
\begin{equation*}
%\label{delta}
\delta (z):=[1-(1-m^{\frac q2})z^{\frac q2}]^\frac{2}{q}
\end{equation*}
and
\begin{equation*}
%\label{changingvariables}
\ell(z):=\frac{mz}{\delta (z)}.
\end{equation*}
It holds that $\ell (0)=0$, $\ell (1)=1$ and
\[
\ell'(z):=\frac{m}{\delta (z)^{\frac q2+1}}.
\]
Since the function $\ell$ is strictly increasing, keeping the change of variables $y=\ell(z)$ into account, inequality \eqref{Gobbino_int} follows if we prove that
\begin{equation*}
\begin{split}
-\frac q {2p}\bigint _{0}^1 \frac{ 1-m^{\frac q2}z^{\frac q2}\delta (z)^{-\frac q2}}{\displaystyle \left(m^\frac q2+(1-m^\frac q2)m^{\frac q2}z^{\frac q2}{\delta (z)^{-\frac q2}}- \displaystyle m^qz^q{\delta (z)^{-q}}\right)^\frac{p+1}{p}}\cdot\frac{m}{\displaystyle \delta(z)^{\frac q2+1}} \ dz\\
+ \bigint_{0}^1\frac{ \left(1 -\frac q{2p}\right)m(1-z^{\frac q2}) +\left(1-\frac{q}{p}\right)m^{\frac q2 +1}(z^{\frac q2}-z^q) }{\displaystyle \left(m^\frac q2-(1-m^\frac q2)m^{\frac q2}z^{\frac q2}-m^qz^q\right)^\frac{p +1}{p}} \ dz\ge 0,
\end{split}
\end{equation*}
Since it is easily checked that $1- z^\frac q 2=\delta (z)^\frac q2 -m^{\frac q2}z^{\frac q2}$ and 
\begin{multline*}
%\label{equivalently1}
\displaystyle  m^\frac q2-(1-m^\frac q2)m^{\frac q2}z^{\frac q2}-m^qz^q   =
\delta(z)^q \left(m^\frac q2+(1-m^\frac q2)\frac{m^{\frac q2}z^{\frac q2}}{\delta (z)^{\frac q2}}- \frac{\displaystyle m^{q}z^q}{\delta (z)^q}\right),
\end{multline*}
the conclusion follows.% and hence the result follows. %Therefore the integrand function is nonnegative and $H(\cdot,p,\frac p2)$ is decreasing with respect to $m$.  
\end{proof}

The two previous Lemmata lead to the following estimates for the function $H$. 
\begin{prop}
\label{prop_stime_H}
 Let $p,q,r>1$ be such that $\frac{2p}{p+2}\le q\leq p$.
 \begin{itemize}
 \item[(i)] If $\frac q2+1\le r\le q+\frac q p$,  then  
\[
H\left(m,p,q,r \right)\ge\pi_{p,q}
\]
for any $m\in[0,1]$.
 \item[(ii)] If $\frac q2+1 < r\le q+\frac q p$,  then  
\[
H\left(m,p,q,r \right)=\pi_{p,q}
\]
if and only if $m=1$.
 \end{itemize}
  \end{prop}
  
\begin{proof}
{\it Case (i).}
If $m=1$, by the definition \eqref{Hdef} of $H$, we have that
\begin{equation}
\label{H=pi}H(1,p,q,r)=2\bigintsss_0^1\frac{dy}{(1-y^q)^\frac1p}=\pi_{p,q}.%\frac {p'}q\tilde\Lambda (p,q,r)%\frac{\pi_{p,q}}{(p-1)^\frac 1p}.
\end{equation}
If $m=0$, it is easily seen that
\begin{equation}\label{H0r}
H(0,p,q,r)=\bigintsss_0^1\frac{dy}{(y^{r-1}-y^q)^\frac1p}\geq \bigintsss_0^1\frac{dy}{(1-y^q)^\frac1p}=\pi_{p,q}.
\end{equation}
When $0<m<1$, since $\bar r:=\frac q2 +1\geq\frac 12+\frac q2 +\frac q{2p}$, exploiting first the  the monotonicity with respect to $r$ (\Cref{prop-monotonia-r}) and then the monotonicity with respect to $m$ (\Cref{Hdecr}), we obtain
 \begin{equation}
 \label{Hnr}
 H(m,p,q,r)\geq K(m) \ge K(0)=\int_0^1\frac{dy}{(y^{\frac q 2}-y^q)^\frac1p}\geq \pi_{p,q},
  \end{equation}
  where $K(m)=H(m,p,q,\bar r)$ is the function defined in \eqref{Kdef}.
  
{\it Case (ii).} The sufficient condition follows by \eqref{H=pi}, while the necessary condition follows by observing that  both the inequality in \eqref{H0r} when $m=0$ and the first inequality in \eqref{Hnr} when $m\in (0,1)$ are strict.
\end{proof}
%\begin{rem}Let us observe that the change of variables $t=(y^{-\frac q2}-1)^\frac 1p$ and  $t=(y^{-q}-1)^\frac 1p$, leads to \begin{equation}\label{H0q}H\left(0,p,q,\frac q2+1\right)=\int_0^1\frac{dy}{[y^{\frac q2}-y^q]^\frac1p}\ge\int_0^1\frac{dy}{[1-y^q]^\frac1p}=H\left(1,p,q,\frac q2+1\right).\end{equation}Therefore, by \eqref{H0r} and \eqref{H0q}, we have\[H(0,p,q,r)>H(1,p,q,r)=\pi_{p,q}.\]\end{rem}

%\section{The sign-changing minimizers}\label{sec_sign}
\section{Proof of the main Theorems}
\label{sec_proof}
A key role in the proof of the main results is played by the sign-changing minimizers. When such solutions occur, both the eigenvalue and the corresponding eigenfunctions can be represented in terms of the function $H$ introduced in the previous section.
%We firstly state some properties of the minimizers that change sign.
\begin{prop}
\label{cambiosegno0}
Let $p,q,r >1$ be such that $\frac{2p}{p+2}\le q\le p$ and $\frac q2+1%\frac 12+\frac q2 +\frac q{2p}
\le r \le %p+1
q+\frac q p$ and suppose that there exists $\alpha>0$ such that $\lambda_\alpha(p,q,r)$ admits a minimizer $y$ that changes sign in $[-1,1]$. Then the following properties hold.
\begin{enumerate}
 \item[{\it (i)}] The minimizer $y$ has exactly one maximum point $\eta_{M}$ in $[-1,1]$, has exactly one minimum point $\eta_{m}$ in $[-1,1]$ and, up to a multiplicative constant, satisfies
\begin{equation}
\label{nor_M_m}
% ||y||_q=1,\quad y(\eta_{M})=M=\max_{[-1,1]} y(x),\quad y(\eta_{ m})= -  m=\min_{[-1,1]} y(x),\quad \text{with}\ \frac{ m}M\in]0,1].
y(\eta_{M})=1=\max_{[-1,1]} y(x),\quad y(\eta_{ m})= -  m=\min_{[-1,1]} y(x),\quad \text{with}\ m\in]0,1].
\end{equation}
 \item[{\it (ii)}] If $y_{+}\ge0$ and $y_{-}\le 0$ are, respectively, the positive and negative part of $y$, then $y_{+}$ and $y_{-}$ are, respectively, symmetric about $x=\eta_{M}$ and $x=\eta_{ m}$.
 \item[{\it (iii)}] There exists a unique zero of $y$ in $]-1,1[$.
 \item[{\it (iv)}] The following representations hold
\[
\begin{split}
\lambda_\alpha(p,q,r)&=\frac q {p' }||y||_q^{q-p}H^p\left(m,p,q,r\right),\\
||y||_q&=\left[\frac{r-1+p'}{q+p'}\gamma+(1-R(m,q,r))\frac{2p'}{p'+q}\right]^\frac 1 q,
\end{split}
\] 
where $\lambda$, $H$ and $R$ have been defined in \eqref{operat}, \eqref{Hdef} and \eqref{Rdef}, respectively.
 \item[{\it (v)}] $
\lambda_\alpha(p,q,r)=\lambda_T(p,q,r).
$ 
\end{enumerate}
\end{prop}
%\begin{rem}
%If $\alpha\le 0$, for any $1\le r \le 2$ the minimizers of $\lambda(\alpha,r)$ has constant sign. Hence in this case the above proposition cannot be applied. 
%\end{rem}
\begin{proof}
For the sake of simplicity, throughout the proof, we will write $\lambda=\lambda_\alpha(p,q,r)$.
%We can multiply the sign-changing minimizer $y$ of $\lambda$ times a suitable (positive or negative) constant such that \eqref{nor_M_m} is verified.

Multiplying the equation in \eqref{el} by $y'$ and integrating in $] -1,1[$, we get
\begin{equation}
\label{inel1d}
\frac{1}{p'} |y'|^p + \frac{\lambda ||y||_q^{p-q}}{q}|y|^q = \frac{ \alpha |\gamma|^{\frac p {r-1}-2}\gamma }{r-1} |y|^{r-2}y + c\qquad\
%\frac{|y'|^p}{p'}+\lambda||y||_q^{p-q} \frac{|y|^q}{q} = \frac{\alpha|\gamma|^{\frac p {r-1}-2}\gamma}{r-1} |y|^{r-2}y + c\qquad\ \text{in } ] -1,1[,
\end{equation}
for a suitable constant $c$, where $\frac1p + \frac{1}{p'}=1$. 

Therefore, since $y'(\eta_{M})=0$ and $y(\eta_{M})=1$, $y'(\eta_{ m})=0$ and $y(\eta_{ m})=-m$, we have
\begin{equation*}
c=\frac{\lambda  ||y||_q^{p-q} }q  -\frac{\alpha|\gamma|^{\frac p {r-1}-2}\gamma}{r-1} = \frac{\lambda  ||y||_q^{p-q}}{q}m^q +  \frac{\alpha|\gamma|^{\frac p {r-1}-2} \gamma }{r-1} m^{r-1} .
%c=\frac{\lambda ||y||_q^{p-q}}{q}M^q-\frac{\alpha|\gamma|^{\frac p {r-1}-2}\gamma}{r-1} M^{r-1}=\frac{\lambda ||y||_q^{p-q}}{q} m^q+\frac{\alpha|\gamma|^{\frac p {r-1}-2}\gamma}{r-1}  m^{r-1}.
\end{equation*}
Hence, we obtain
\begin{equation}\label{costnl}
\left\{
\begin{array}{l}
\frac{\alpha\displaystyle |\gamma|^{\frac p {r-1}-2}\gamma}{ \displaystyle r-1} =\frac{\displaystyle \lambda  ||y||_q^{p-q}}{\displaystyle  q} R\left( m,q,r\right)\\[.3cm]
c=\frac{\lambda  ||y||_q^{p-q}}{\displaystyle q}\left (1-R\left(m,q,r\right)\right).
\end{array}
\right.
\end{equation}
%\begin{equation}\label{costnl}\left\{\begin{array}{l}|\gamma|^{\frac p {r-1}-2}\gamma =\frac{\displaystyle \lambda ||y||_q^{p-q} (r-1)}{\displaystyle \alpha q} Z(m,M,q,r)\\[.3cm]c=\frac{\ds\lambda ||y||_q^{p-q}}{\displaystyle q}(M^q-M^{r-1}Z(m,M,q,r))\end{array}\right.\end{equation}where\begin{equation*}Z(m,M,q,r)=\frac{M^q-m^q}{M^{r-1}+m^{r-1}}.\end{equation*}
%and\begin{equation*}T(m,M,q,r)=\frac{M^qm^{r-1}+M^{r-1}m^q}{M^{r-1}+m^{r-1}}=M^q-M^{r-1}Z(m,M,q,r).\end{equation*}
%Now let us observe that the function $R$ is positive for $m\in]0,1[$ and negative for $m>1$. As a consequence, by the first equation of \eqref{costnl}, we can say that the $r$-average is positive for $m\in]0,1[$ and negative for $m>1$. Throughout this Section we will assume that \[\bar m=-\min_{[-1,1]} y(x)\in\ ]0,1].\] Indeed if we consider only the case of $m\in]0,1[$, indeed if we consider $-y$ instead of $y$, then it satisfies the same equations \eqref{el} and \eqref{inel1d}. 
So, equation \eqref{inel1d} can be written as 
\begin{equation}\label{integratedELconstant}
\frac{1}{p'} |y'|^p + \frac{\lambda  ||y||_q^{p-q} }q  |y|^q  = \frac{\lambda  ||y||_q^{p-q} }{q} R\left(m,q,r\right) |y|^{r-2}y + \frac{\lambda  ||y||_q^{p-q} }{q} \left (1-R\left(m,q,r\right)\right),
\end{equation}
that is
\begin{equation}
\label{integratedel}
|y'|^p=\frac{p' \lambda   ||y||_q^{p-q}}{q}\left(1-R\left( m,q,r\right)(1- |y|^{r-2}y) - |y|^q)\right).
\end{equation}
%and hence\begin{equation}\label{integratedel}\left|\frac{y'}{M}\right|^p=\frac{\lambda p' }{q}M^{q-p}\left(1-R\left(\frac mM,q,r\right)\left(1- \left|\frac y M\right|^{r-2}\frac yM\right) - \left|\frac yM\right|^q\right) \qquad\ \text{in}  ]-1,1[.\end{equation}
It is easy to see that the number of zeros of $y$ have to be finite, let us denote them by 
\[
-1=\zeta_{1}<\ldots<\zeta_{j}<\zeta_{j+1}<\ldots<\zeta_{n}=1
\] 
be the zeroes of $y$. Moreover, as in \cite{CD}, we observe that
\begin{equation}
\label{dac}
y'(x)=0 \iff y(x)=- m\text{ or }y(x)=1.
\end{equation}
Indeed, if we set
\[
\mu (y):=1-R\left(m,q,r\right)(1-|y|^{r-2}y)-|y|^{q},\quad y\in \left[- m,1\right],
\]
then \eqref{integratedel} gives
\begin{equation}
\label{CDproof}
\left| y'\right|^{p}=\frac{p' \lambda  ||y||_q^{p-q} }{q} \mu\left( y\right).
%\left|\frac {y'}M\right|^{p}=\frac{p' \lambda }{q}M^{q-p}\ \mu\left(\frac yM\right).
\end{equation}
%By substitution $z=y/M$, we have\begin{equation}\left| {z'}\right|^{p}=\frac{p' \lambda}{q}M^{q-p}g(z).\end{equation}
Let us observe that $\mu(- m)=\mu(1)=0$. Being $q\ge r-1$ by assumption, it is easily seen that for any $\overline y$ such that $\mu'(\overline y)=0$ then $\mu(\overline y)>0$. Hence, $\mu$ does not vanish in $]-m,1[$ and, therefore, by \eqref{CDproof}, $y'(x)\ne 0$ if $y(x)\ne 1$ and ${y(x)}\ne - m$, that proves \eqref{dac}.

This implies that $y$ has no other local minima or maxima in $]-1,1[$, that in any interval $]\zeta_{j},\zeta_{j+1}[$ where $y>0$ there is a unique maximum point and that in any interval $]\zeta_{j},\zeta_{j+1}[$ where $y<0$ there is a unique minimum point. 

Then the properties \textit{(i)}, \textit{(ii)} and \textit{(iii)} follow by adapting the argument of \cite[Lemma 2.6]{DGS}, see also \cite{DP} for the case $p=2$. 
We remark that they can also be proved by using a symmetrization argument, by rearranging the functions $y^{+}$ and $y^{-}$ and using the P\'olya-Szeg\H o inequality and the properties of rearrangements (see also, for example, \cite{BFNT} and \cite{D}).
Specifically, one can prove that
\begin{itemize}
\item any interval $]\zeta_{j},\zeta_{j+1}[$ given by two subsequent zeros of $y$, and in which $y=y^{+}>0$, has the same length; any of such interval, $y^{+}$ is symmetric about $x=\frac{\zeta_{j}+\zeta_{j+1}}{2}$;
\item any interval $]\zeta_{j},\zeta_{j+1}[$ given by two subsequent zeros of $y$, and in which $y=y^{-}<0$, has the same length; any of such interval, $y^{-}$ is symmetric about $x=\frac{\zeta_{j}+\zeta_{j+1}}{2}$;
\item there is a unique zero of $y$ in $]-1,1[$.
\end{itemize}

In order to show {\it (iv)}, it is not restrictive to suppose the order relation $\eta_M<\eta_m$ between the unique maximum and the unique minimum point of $y$. It is easily seen (\cite[Lem. 2.6]{DGS}) that $\eta_{M}-\eta_{m}=1$, with $y'<0$ in $]\eta_{M},\eta_{m}[$. Then, from \eqref{integratedel}, we have
\begin{equation*}
\frac{-y'}{[1-R\left(m,q,r\right)(1-|y|^{r-2}y)- |y|^q]^\frac 1 p}=\left({\frac{p'\lambda  ||y||_q^{p-q} }{q}}\right)^\frac 1 p  \quad\ \text{in} \ ]\eta_{M},\eta_{ m}[.
%\frac{-z'}{[1-R\left(\frac mM,q,r\right)(1-|z|^{r-2}z)- z^q]^\frac 1 p}=\left({\frac{p'\lambda}{q}}\right)^\frac 1 p M^{\frac qp-1} \quad\ \text{in} \ ]\eta_{M},\eta_{ m}[.
\end{equation*}
%We write
%\begin{equation*}
%\frac{|y'|}{ [1-z(m,r)(1-|y|^{r-1}y) - y^2]^{1/2}}=\lambda^{1/2} \qquad\ \text{on} \ [\eta_1,\eta_2].
%\end{equation*}
Then, integrating between $\eta_{M}$ and $\eta_m$, we have
\begin{equation}
\label{lambdaM}
\lambda=\frac{q}{p'} ||y||_q^{q-p}\left[\bigintsss_{- m}^1\frac{dy}{ [1-R\left( m,q,r\right)(1-|y|^{r-2}y)- |y|^q]^\frac 1 p}\right]^{p} = \frac q {p' }||y||_q^{q-p}H^p\left(m,p,q,r\right),
\end{equation}
that is the first part of {\it (iv)}. The second part follows by integrating \eqref{integratedELconstant} over $(-1,1)$ and recalling that $||y'||_p^p+\alpha |\gamma|^{\frac p{r-1}}=\lambda||y||_q^p$.

Finally, since by \Cref{eig_prop} we know that $\lim_{\alpha\to +\infty}\lambda_\alpha(p,q,r)=\lambda_T(p,q,r)$ and since the relation \eqref{lambdaM} does not depend on $\alpha$, we have
 \[
  \frac q {p' }||y||_q^{q-p}H^p\left(m,p,q,r\right)=
\lim_{\alpha\to +\infty}\lambda_\alpha(p,q,r)= \lambda_T(p,q,r),
\]
that gives {\it (v)}.
\end{proof}

At this stage, we are in a position to state that each sign-changing minimizer of problem \eqref{operat} is a symmetric and zero-average function. 

\begin{prop}\label{cambiosegno} 
Let $p,q,r >1$ be such that $\frac{2p}{p+2}\le q\le p$ and suppose that there exists $\alpha>0$ such that $\lambda_\alpha(p,q,r)$ admits a minimizer $y$ that changes sign in $[-1,1]$ and satisfies the conditions in  \eqref{nor_M_m}.
\begin{itemize}
\item[{\it (i)}] If $\frac q2 +1< r \le q+\frac qp$, then
\begin{enumerate}
\item[{\it (a)}] %If $\frac q 2+1< r \le q+1$, then\begin{equation}\label{r-average}
$\int_{-1}^{1}|y|^{r-2}y\,dx=0;$%\end{equation}
\item[{\it (b)}] %If $\frac q 2+1\le r \le q+1$ and \eqref{r-average} holds, then 
{$y(x)=C\sin_{p,q} (\lambda_T(p,q,r) x)$, with $C\in \R\setminus\{0\}$};
\item[{\it (c)}] the only point $\bar x\in ]-1,1[$ where $y$ vanishes is $\overline x=0$. 
\end{enumerate}
\item[{\it (ii)}] If $r=\frac q2 +1$ and $\int_{-1}^{1}|y|^{r-2}y\,dx=0$, then {$y(x)=C\sin_{p,q} (\lambda_T(p,q,r) x)$, with $C\in \R\setminus\{0\}$},
and the only point in $\bar x\in ]-1,1[$ where $y$ vanishes is $\overline x=0$. 
\end{itemize}
\end{prop}
\begin{proof} Since the exact value of the best constant in the Twisted inequality is known from \cite[Thm. 1.1]{CD} and, notably, it does not depend on the parameter $r$, \Cref{cambiosegno0}\textit{(iv)} and \textit{(v)} yield
\begin{equation}
\label{vero_twisted}
\begin{split}%\left(\frac{2}{b-a}\right)^{\frac 1{p'}+\frac 1q}
\lambda_T(p,q,r) &= \left[ \left(\frac 1 {p'}\right)^\frac 1q\left(\frac 1q \right)^{\frac 1 {p'}}\left(\frac 2 {p'+q}\right)^{\frac 1p-\frac 1q}q \right]^p\pi_{p,q}^p=\frac {q}{p'}\left(\frac {2p'} {p'+q}\right)^{1-\frac pq} \pi_{p,q}^p\\
&\leq \frac q {p' }\left[\frac{r-1+p'}{q+p'}\gamma+(1-R(m,q,r))\frac{2p'}{p'+q}\right]^{1-\frac p q}H^p\left(m,p,q,r\right)\\
&= \lambda_\alpha(p,q,r)= \lambda_T(p,q,r).
\end{split}
 \end{equation}
Hence, since by \Cref{prop_stime_H}{\it (ii)} we know that $H\left(m,p,q,r \right)=\pi_{p,q}$ if and only if $m=1$, the strict  decrease of $R$ with respect to $m$ and the first identity in \eqref{costnl} give that 
\begin{equation}
\label{r-average}
\int_{-1}^{1} |y|^{r-2}ydx=0,
\end{equation}
that is \textit{(a)}. To prove \textit{(b)} and \text{(c)}, let us explicitly observe that, when \eqref{r-average} holds, then $y$ solves problem \eqref{eig_pq_1D} with $\lambda=\lambda_T(p,q,r) ||y||_q^{p-q}$. Hence $y(x)=C\sin_{p,q} (\pi_{p,q} x)$, with $C\in \R\setminus\{0\}$.  

The case {\it (ii)} easily follows using the same arguments.
%  \eqref{inel1d} with $\gamma=0$ becomes
%\[
%\frac{y'^2}{2}+\lambda\frac{y'^2}{2}=\frac{\lambda}{2}\quad\text{and hence}\quad \frac{|y'|}{\sqrt{1-y^2}}=\sqrt{\lambda}.
%\]
%By integration, supposing without loss of generality that $y>0$ in $]-1,\bar x[$, we have
%\[
%\int_{\frac{\bar x+1}{2}}^{\bar x} \frac{-y'}{\sqrt{1-y^2}}dx=\sqrt{\lambda}\left(\bar x-\frac{\bar x +1}{2}\right)\quad\text{and}\quad\int^{\frac{1-\bar x}{2}}_{\bar x} \frac{-y'}{\sqrt{1-y^2}}dx=\sqrt{\lambda}\left(\frac{1-\bar x}{2}-\bar x\right),
%\]
%moreover, by a change of variables,
%\[
%\int_0^1 \frac{dy}{\sqrt{1-y^2}}=\sqrt{\lambda}\left(\bar x-\frac{\bar x +1}{2}\right)\quad\text{and}\quad\int_{-1}^0 \frac{dy}{\sqrt{1-y^2}}=\sqrt{\lambda}\left(\frac{1-\bar x}{2}-\bar x\right).
%\]
%Hence $\bar x=0$ is the only point where $y$ vanishes and the minimizer $y$ is odd}. Therefore the Proposition is completely proved.
\end{proof}
%Now we are in position to prove that if $\alpha$ is sufficiently large, there exists $\alpha$ for which $\lambda(\alpha,q)$ admits a changing-sign solution.

At this stage, we are in a position to prove the main Theorems of the paper.
\begin{proof}[Proof of \Cref{mainthm1}]
When $\alpha\le 0$, the minimizers of \eqref{operat} have constant sign; indeed
\[
\mathcal{Q}_\alpha[u] \ge \mathcal Q_{\alpha}[|u|],
\]
with equality if and only if $u\ge 0$ or $u\le 0$. 

In order to prove the main result, we will show that there exists $\alpha>0$ for which problem \eqref{operat} admits a minimizer $y$ that changes sign. Suppose, by contradiction, that for any $k\in\N$, there exists a divergent sequence $\alpha_{k}$, with a corresponding sequence of nonnegative eigenfunctions $\{y_{k}\}_{k\in\N}$ associated to $\lambda_{\alpha_{k}}(p,q,r)$, such that % $\int_{-1}^1y_{k}|y_{k}|^{r-2}dx>0$ 
 $\|y_k\|_q=1$. 

By \Cref{eig_prop}, we have that $\lambda_{\alpha_{k}}(p,q,r)\le\lambda_T(p,q,r)$ and hence, it holds that
\begin{equation}
\label{contrad}
\displaystyle \int_{-1}^{1} | y_k'|^p\  dx+\alpha_k \left(\int_{-1}^{1} y_k^{r-1}\  dx\right)^{\frac{p}{r-1}}\le\lambda_T(p,q,r).
\end{equation}
Therefore, $y_k$ converges (up to a subsequence) to a function $y\in W_0^{1,p}(-1,1)$, strongly in $L^p(-1,1)$ and weakly in $W_0^{1,p}(-1,1)$. Moreover $\|y\|_{p}=1$ and $y$ is not identically zero. As a consequence, $\|y\|_{r-1}>0$ and, letting $\alpha_k\rightarrow +\infty$ in \eqref{contrad}, we have a contradiction. We have thus proved that there exists a positive value of $\alpha$ such that the minimization problem \eqref{operat}
admits an eigenfunction $y$ satisfying $\int_{-1}^1|y|^{r-2}y\  dx=0$. In this case, $\lambda_\alpha(p,q,r)=\lambda_T(p,q,r)$ and, up to a multiplicative constant, $y=\sin_{p,q} (\pi_{p,q} x)$.

%This result will allow us to prove, in the following Section, the existence of a critical value of the parameter for the problem \eqref{operat} such that the minimizers are symmetric above this value. 
Since, by \Cref{propr}, $\lambda_\alpha(p,q,r)$ is nondecreasing and Lipschitz continuous with respect to $\alpha$, we can define 
\[
\alpha_{C}=\min\{\alpha\in\R\colon \lambda_\alpha(p,q,r)=\lambda_T(p,q,r)\}=\sup\{\alpha\in\R\colon \lambda_\alpha(p,q,r)<\lambda_T(p,q,r) \}.
\]
It is easily seen that this critical value of the parameter is strictly positive.
\end{proof}

\begin{proof}[Proof of \Cref{mainthm2}]
If $\alpha<\alpha_{C}$, then all minimizers corresponding to $\lambda_\alpha(p,q,r)$ have constant sign; otherwise, when $\alpha<\alpha_{C}$, we have $\lambda_\alpha(p,q,r)=\lambda_T(p,q,r)$. 

Let us now consider the case $\alpha>\alpha_{C}$, and suppose by contradiction that there exists $\bar \alpha>\alpha_{C}$ and a minimizer $\bar y$ such that 
\[
\int_{-1}^{1}|\bar y|^{r-2}\bar y\,dx>0,\quad \|y\|_{p}=1, \mathcal Q_{\bar\alpha}[\bar y]=\lambda_{\bar\alpha}(p,q,r).\]
Then, for any $\varepsilon>0$ sufficiently small, we compute 
\begin{align*}
\mathcal Q_{\bar\alpha-\eps}[\bar y] &=
\mathcal Q_{\bar\alpha}[\bar y]-\eps\left(\int_{-1}^{1}|\bar y|^{r-2}\bar y\,dx\right)^{\frac p{r-1} }\\ &=\lambda_{\bar\alpha}(p,q,r)-\eps\left(\int_{-1}^{1}|\bar y|^{r-2}\bar y\,dx\right)^{\frac p{r-1}}<\lambda_{\bar\alpha}(p,q,r).
\end{align*}
This implies that  
\[
\lambda_{T}(p,q,r)=\lambda_{\alpha_{C}}(p,q,r)\le\lambda_{\bar\alpha-\eps}(p,q,r)<\lambda_{\bar\alpha}(p,q,r),\]
which contradicts the definition of $\alpha_C$. Hence, for all $\alpha>\alpha_C$, any minimizer $y$ satisfy
\[
\int_{-1}^{1}|y|^{r-2}y\,dx=0.
\]
 Hence, by \Cref{cambiosegno}, the claims {\it (i)} and {\it (ii)} follows. Regarding {\it (iii)}, it is not difficult to show, using suitable approximating sequences, that $\lambda_{\alpha_{C}}(p,q,r)$ admits both a nonnegative minimizer and a minimizer with vanishing $r$-average. 
\end{proof}

\begin{rem}%To conclude the proof of Theorem \ref{mainthm2}, we have 
To analyze the behavior of the solutions when $r=p+1$, we consider the case $\alpha=\alpha_C(p,q,p+1)$. In this setting, the corresponding positive minimizer $y$ is a solution of
\begin{equation*}
\left\{
\begin{array}{ll}
(|y'|^{p-2}y')'+\lambda_T(p,q,p+1)||y||_q^{p-q} y^{q-1}=\alpha_C(p,q,p+1) y^{q-1} &\text{in }]-1,1[\\
y(-1)=y(1)=0.
\end{array}
\right.
\end{equation*}
The positivity of the eigenfunction ensures (see also \eqref{vero_twisted}) that
\begin{equation*}
\lambda_T(p,q,p+1)||y||_q^{p-q}-\alpha_C(p,q,p+1)=\lambda_{0}(p,q,p+1)||y||_q^{p-q}=\frac {q}{p'}\left(\frac {2p'} {p'+q}\right)^{1-\frac pq}\left(\frac{ \pi_{p,q}}{2}\right)^p.
\end{equation*}
As a consequence, we deduce the explicit expression
\[
\alpha_{C}(p,q,p+1)=\frac{2^p-1}{2^p}\frac {q}{p'}\left(\frac {2p'} {p'+q}\right)^{1-\frac pq} \pi_{p,q}^p.\]
\end{rem}

\begin{rem}
When the exponents $p,q,r$ satisfy the assumptions of the main Theorems, we obtain the following lower bound for the critical value $\alpha_C(p,q,r)$:
	\begin{equation}
	\label{stimar}
	\alpha_{C}(p,q,r)\ge \frac{2^p-1}{2^{\frac p{r-1}+p-1}}\frac {q}{p'}\left(\frac {2p'} {p'+q}\right)^{1-\frac pq} \pi_{p,q}^p.
	\end{equation}
To derive estimate \eqref{stimar}, we exploit the monotonicity of $\lambda_{\alpha}(p,q,r)$ with respect to $\alpha$, and test the Raylegh quotient with the function $u(x)=\sin_{p,q}(\frac{\pi_{p,q}}{2} (x+1))$. Hence
	\begin{equation*}
	\begin{split}
\lambda_T(p,q,r)&=\lambda_{\alpha_{C}}(p,q,r)\le\mathcal{Q}[u,\alpha_{C}]=\frac {q}{p'}\left(\frac {2p'} {p'+q}\right)^{1-\frac pq} \left(\frac{\pi_{p,q}}{2}\right)^p+\alpha_{C}\left(\int_{-1}^{1}u^{r-1}dx\right)^{\frac p{r-1}}\\
& \le \frac {q}{p'}\left(\frac {2p'} {p'+q}\right)^{1-\frac pq} \left(\frac{\pi_{p,q}}{2}\right)^p +\alpha_{C}2^{\frac p{r-1}-1}.
	\end{split}
	\end{equation*}
\end{rem}
%If $v$ is a positive eigenfunction of \eqref{operat} with $\alpha=\alpha_{q}$, with $\|v\|_{2}=1$, then
%\[
%\pi^{2}=\mathcal{Q}[v,\alpha_{q}]\Rightarrow \alpha_{q}= \frac{\pi^{2}-\int_{-1}^{1}(v')^{2}dx}{\left(\int_{-1}^{1}v^{q}dx\right)^{2/q}}
%\]

\subsection*{Acknowledgement} 
This work has been partially supported by GNAMPA of INdAM and by the Italian Ministry of University and Research through the PRIN 2022 research project ``Partial differential equations and related geometric-functional inequalities'' (grant number 20229M52AS).


\begin{thebibliography}{9999999}

%\bibitem[AB]{AB} W. Allegretto, A. Barabanova. "Positivity of solutions of elliptic equations with nonlocal terms". Proc. Roy. Soc. Edinburgh Sect. A \textbf{126} (1996): 643-663.

\bibitem[BB]{BB} L. Barbosa and P. B\'erard, \textit{Eigenvalue and \lq\lq twisted\rq\rq\ eigenvalue problems, applications to CMC surfaces}, J. Math. Pures Appl. (9) {\bf 79} (2000): 427--450.

\bibitem[BK]{BK} M. Belloni, B. Kawohl, \textit{A symmetry problem related to Wirtinger's and Poincar\'e's inequality}. J. Differential Equations \textbf{156.1} (1999): 211--218.

\bibitem[BDNT]{BDNT} B. Brandolini, F. Della Pietra, C. Nitsch, C. Trombetti, \textit{Symmetry breaking in a constrained Cheeger type isoperimetric inequality}. ESAIM Control Optim. Calc. Var. \textbf{21.2} (2015): 359--371.

\bibitem[BFNT]{BFNT} B. Brandolini, P. Freitas, C. Nitsch, C. Trombetti, \textit{Sharp estimates and saturation phenomena for a nonlocal eigenvalue problem}. Adv. Math. \textbf{228.4} (2011): 2352--2365.

%\bibitem[B]{B} A. M. Bresquar, \textit{Sulla diseguaglianza di Wirtinger}. Rend. Semin. Mat. Univ. Padova {\bf 51} (1974): 257--268.

\bibitem[BCGM]{BCGM} F. Brock, G. Croce, O. Guib\'e, A. Mercaldo, \textit{Symmetry and asymmetry of minimizers of a class of noncoercive functionals}. Adv. Calc. Var. \textbf{13.1} (2020): 15--32.

\bibitem[BKN]{BKN} A. P. Buslaev, V. A. Kondrat'ev, A. I. Nazarov, \textit{On a family of extremal problems and related properties of an integral}, Mat. Zametki \textbf{64.6} (1998): 830-838 (Russian); English transl.: Math. Notes \textbf{64.5-6} (1998): 719--725. 

\bibitem[CD]{CD} G. Croce, B. Dacorogna, \textit{On a generalized Wirtinger inequality}, Discrete Contin. Dyn. Syst. \textbf{9.5} (2003): 1329--1341. 

%\bibitem[CHP]{CHP} G. Croce, A. Henrot, G. Pisante, \textit{An isoperimetric inequality for a nonlinear eigenvalue problem}, Ann. lnst. H. Poincar\'e Non Lin\'eaire \textbf{29} (2012): 21--34; \textit{Corrigendum to ``An isoperimetric inequality for a nonlinear eigenvalue problem''},  Ann. lnst. H. Poincar\'e Non Lin\'eaire \textbf{32} (2015): 485--487.

\bibitem[DGS]{DGS} B. Dacorogna, W. Gangbo, N. Sub\'ia, \textit{Sur une g\'en\'eralisation de l'in\'egalit\'e de Wirtinger}.  Ann. lnst. H. Poincar\'e Non Lin\'eaire \textbf{9.1} (1992): 29--50.

\bibitem[DP]{D} F. Della Pietra, \textit{Some remarks on a shape optimization problem}. Kodai Math. J. \textbf{37} (2014): 608--619.

%\bibitem[DG]{DG} F. Della Pietra, N. Gavitone, "Sharp bounds for the first eigenvalue and the torsional rigidity related to some anisotropic operators". Mathematische Nachrichten \textbf{287} 2-3 (2014): 194-209.

\bibitem[DPP1]{DP} F. Della Pietra, G. Piscitelli, \textit{A saturation phenomenon for a nonlinear nonlocal eigenvalue problem}. NoDEA Nonlinear Differential Equations Appl. \textbf{23.6} (2016): 1--18.

\bibitem[DPP2]{DP2} F. Della Pietra, G. Piscitelli. \textit{Saturation phenomena for some classes of nonlinear nonlocal eigenvalue problems}. Atti Accad. Naz. Lincei Rend. Lincei Mat. Appl. \textbf{31} (2020): 131--150.

\bibitem[Du]{Du} J.-M.-C. Duhamel, \textit{Second m\'{e}moire sur les ph\'{e}nom\'{e}nes thermo-m\'{e}caniques}, J. \'{E}c. polytech. Math. \textbf{15.25} (1837): 1--57.

\bibitem[E]{E} Y. V. Egorov, \textit{On a Kondratiev problem}. C. R. Math. Acad. Sci. Paris \textbf{324} (1997): 503--507.

%\bibitem[EK]{EK} Y. V. Egorov, V. A. Kondratiev. "On a Lagrange problem". C.R.A.S. Paris Ser. 1 \textbf{317} 9 (1993): 903-908.

\bibitem[F1]{F1} P. Freitas, \textit{A nonlocal Sturm-Liouville eigenvalue problem}, Proc. Roy. Soc. Edinburgh Sect. A {\bf 124} (1994), 169--188.

\bibitem[F2]{F} P. Freitas, \textit{Nonlocal reaction-diffusion equations}. Diff. Eq. Appl. Biol., Halifax, NS (1997). Fields Inst. Commun. \textbf{21}, Amer. Math. Soc., Providence, RI (1999): 187--204.

\bibitem[FH]{FH} P. Freitas, A. Henrot, \textit{On the first twisted Dirichlet eigenvalue}. Comm. Anal. Geom. \textbf{12} (2004): 1083--1103.

\bibitem[FV]{FV} P. Freitas, M. Vishnevskii, \textit{Stability of stationary solutions of nonlocal reaction-diffusion equations in $m$-dimensional space}, Differential Integral Equations {\bf 13} (2000), 265--288.

\bibitem[GN]{GN} I. V. Gerasimov, A. I. Nazarov, \textit{Best constant in a three-parameter Poincar\'e inequality}. Probl. Mat. Anal.  \textbf{61} (2011): 69-86 (Russian). English transl.: J. Math. Sci. \textbf{179} 1 (2011): 80--99. 

\bibitem[GGR]{GGR} M. Ghisi, M. Gobbino, G. Rovellini. \textit{Symmetry-breaking in a generalized Wirtinger inequality}. ESAIM Control Optim. Calc. Var.  \textbf{24.4} (2018): 1--14.

%\bibitem[K]{K} B. Kawohl, Symmetry results for functions yielding best constants in Sobolev-type inequalities. Discrete Contin. Dynam. Syst. 6 (2000) 683-690.

%\bibitem[KN]{KN} N. Kuznetsov, A. Nazarov. \textit{Sharp constants in the Poincar\'e, Steklov and related inequalities (a survey)}. Mathematika \textbf{61.2} (2015): 328--344.

\bibitem[LE]{LE} J. Lang, D. Edmunds. \textit{Eigenvalues, Embeddings and Generalised Trigonometric Functions}. Springer (2011): xi + 220.

\bibitem[Lin]{L} P. Lindqvist, \textit{Some remarkable sine and cosine functions}, Ric. Mat. \textbf{44} (1995): 269--290.

\bibitem[Lio]{Lio} J. Liouville, \textit{Solution nouvelle d'un probl\`{e}me d'Analyse r\'{e}latif aux ph\'{e}nom\`{e}nes thermo-m\'{e}caniques}, J. Math. Pure Appl. {\bf 2} (1837): 439--456.

\bibitem[N1]{N1} A. I. Nazarov, \textit{On exact constant in the generalized Poincar\'e inequality}. Probl. Mat. Anal. \textbf{24} (2002): 155-180 (Russian). English transl.: J. Math. Sci. \textbf{112.1} (2002): 4029--4047.

\bibitem[N2]{Narxiv} A. I. Nazarov, \textit{On symmetry and asymmetry in a problem of shape optimization}, Ar$\chi$iv:1208:3640 (2012).

\bibitem[Pe]{Pe} J. Peetre, \textit{The best constant in some inequalities involving $L_q$ norms}, Ric. Mat. 21 (1972): 176--183.

\bibitem[Pin]{P} R. Pinsky, \textit{Spectral analysis of a class of non-local elliptic operators related to Brownian motion with random jumps}. Trans. Amer. Math. Soc. \textbf{361} (2009): 5041--5060. 

\bibitem[Pis]{Pi} G. Piscitelli, \textit{A nonlocal anisotropic eigenvalue problem}. Differential Integral Equations \textbf{29.11/12} (2016): 1001--1020.

%\bibitem[\^O]{O} M. \^Otani, \textit{A remark on certain nonlinear elliptic equations}. Proceedings of the Faculty of Science, Tokai University, \textbf{19} (1984): 23-28.

\bibitem[S]{S} R. P. Sperb, \textit{On an eigenvalue problem arising in chemistry}. Z. Angew. Math. Phys. \textbf{32} 4 (1981): 450--463.

\end{thebibliography}
\end{document}